\documentclass[preprint]{elsarticle}

\usepackage{dutchcal}
\usepackage{dsfont}
\usepackage{bbold} 
\usepackage[font=small,labelfont=bf]{caption}
\usepackage{xcolor}
\usepackage[normalem]{ulem}
\usepackage{tabularx}
\journal{European Journal of Combinatorics}

\usepackage{amsthm}
\usepackage{amsmath}
\usepackage{amssymb}
\usepackage{epigraph}
\usepackage{graphicx}
\usepackage{amsmath}
\usepackage{subfigure}
\usepackage{xspace}
\usepackage{color, soul}

\usepackage[colorinlistoftodos,prependcaption,textsize=tiny]{todonotes}

\newcommand{\Fraisse}{Fra\"{i}ss\'{e}\xspace} 

\newcommand{\N}{\mathbb{N}} 

\newcommand{\HH}{{\bf HH}\xspace} 
\newcommand{\MH}{{\bf MH}\xspace} 
\newcommand{\MM}{{\bf MM}\xspace} 

\newcommand{\aut}[1]{\mathrm{Aut}(#1)}

\newcommand{\trgn}[1]{{\normalfont(}$\Diamond_{#1}${\normalfont )}\xspace}
\newcommand{\Rad}{\mathcal R}

\newtheorem{theorem}{Theorem}
\newtheorem{lemma}[theorem]{Lemma}
\newtheorem{proposition}[theorem]{Proposition}
\newtheorem{definition}[theorem]{Definition}
\newtheorem{problem}[theorem]{Problem}
\newdefinition{remark}{Remark}
\newdefinition{fact}{Fact}
\newdefinition{example}{Example}

\newtheorem{observation}[theorem]{Observation}
\newtheorem{claim}[theorem]{Claim}

\begin{document}
\begin{frontmatter}

\title{Morphism extension classes of countable $L$-colored graphs}

\author[ia]{Andr\'{e}s Aranda}
\address[ia]{Institut f\"ur Algebra, Technische Universit\"{a}t Dresden, Zellescher Weg 12-14, Dresden}
\ead{andres.aranda@gmail.com}

\author[iuuk,cas]{David Hartman}
\address[iuuk]{Computer Science Institute of Charles University, Charles University, Malostransk\'e n\'{a}m. 25, Prague 1}
\address[cas]{Institute of Computer Science of the Czech Academy of Sciences, Pod Vod\'{a}renskou v\v{e}\v{z}\'{i} 271/2, Prague 8}
\ead{hartman@iuuk.mff.cuni.cz}

\begin{abstract}
In~\cite{Hartman:2014}, Hartman, Hubi\v cka and Ma\v sulovi\'c studied the hierarchy of morphism extension classes for finite $L$-colored graphs, that is, undirected graphs without loops where sets of colors selected from $L$ are assigned to vertices and edges. They proved that when $L$ is a linear order, the classes $\MH_L$ and $\HH_L$ coincide, and the same is true for vertex-uniform finite $L$-colored graphs when $L$ is a diamond. In this paper, we explore the same question for countably infinite $L$-colored graphs. We prove that $\MH_L=\HH_L$ if and only if $L$ is a linear order.
\end{abstract}


\end{frontmatter}
\section{Introduction}
A relational structure is called {\em ultrahomogeneous} if any local isomorphism between finite induced substructures can be extended to an automorphism. This property, which can be traced back to Cantor's work, was developed to its current form by \Fraisse~\cite{Fraisse:1953}, and has been studied from model- and group-theoretic perspectives; more recently, homogeneous structures and, more generally, $\omega$-categorical relational structures have been studied in the context of CSPs by Bodirsky~\cite{Bodirsky:2008}. Homogeneous structures possess a high degree of symmetry: intuitively, all finite isomorphic induced substructures are sitting in the same way within the structure; more precisely, the isomorphism type of a finite induced $H\subset M$ determines its $\aut M$-orbit. 

In the case of graphs, ultrahomogeneity can be understood as a generalization of transitivity; for example, vertex- and edge- transitivity can be thought of as ultrahomogeneity restricted to vertices or edges. Such a strong symmetry condition leads in many cases to a very small set of structures satisfying this property. For example, the class of finite graphs contains, up to complementation, only a few ultrahomogeneous elements: the cycle of length $5$, the line graph $L(K_{3,3})$, and disjoint unions of cliques of the same size (Gardiner~\cite{Gardiner:1976}). Moving towards richer structures (i.e., more relations or higher arity) usually results in larger classes of ultrahomogeneous structures, although it is a well-known fact that there are uncountably many infinite ultrahomogeneous directed graphs, see Henson~\cite{Henson:1971}. A different approach is represented by a notion of relational complexity evaluating minimal number of additional relations needed to make a structure ultrahomogeneous~\cite{Cherlin:2016,Hartman:2015}. For such problem we have a structure on input and find a particular way to homogenize maintaining its automorphism group, while in classification problems as those mentioned above we fix a language and search for homogeneous structures.

We consider only relational structures and use $\mathcal L$ to denote their language. If the language is not obvious from the context or we need to stress it, we use standard notion of $\mathcal L$-structure. By changing from the category of $\mathcal L$-structures with embeddings to the category of $\mathcal L$-structures with homomorphisms, we obtain the definition of homomorphism-homogeneity, introduced by Cameron and Ne\v{s}et\v{r}il~\cite{CameronNesetril:2006}. A relational structure is {\em homomorphism-homogeneous} if any local homomorphism can be extended to an endomorphism. This type of homogeneity has been less studied than ultrahomogeneity, partly due to its relatively recent introduction; at the time of this writing, a complete classification of countable homomorphism-homogeneous undirected graphs is still missing. Among existing classifications of homomorphism-homogeneous relational structures one can name the classifications of partially ordered sets~(Ma\v sulovi\'c \cite{Masulovic:2007}, Cameron-Lockett~\cite{CameronLockett:2010}) and transitive tournaments~(Ili\'c, Ma\v{s}ulovi\'{c} and Rajkovi\'{c}~\cite{IlicEtAl:2008}). Classifications themselves have interesting algorithmic aspects from the viewpoint of recognition. If we restrict ourselves to directed graphs with loops allowed we can define a decision problem which resolves whether graphs of this type are homomorphism-homogeneous as follows:

\begin{verbatim}HOMHOM-DIRECTED-LOOPS\end{verbatim}
\noindent {\em Instance:} Any countable directed graph $G$ with loops allowed\newline
\noindent {\em Question:} Decide whether this graph is homomorphism-homogeneous.\newline

This particular problem has been shown to be {\em co-NP} complete by Rusinov and Schweitzer~\cite{RusinovSchweitzer:2010}, but the corresponding decision problem for partially ordered sets can be solved in polynomial time using the classification results. Additionally to this and similarly to the case mentioned above, homomorphism-homogeneity has been studied as a property of templates of CSPs~\cite{CMPech:2011}. 

Already in the original paper of Cameron and Ne\v{s}et\v{r}il~\cite{CameronNesetril:2006}, several types of homomorphism-homogeneity were defined, according to the type of homomorphism used. We say that an $\mathcal L$-structure $G$ belongs to:
\begin{enumerate}
\item $\HH_{\mathcal L}$ if every homomorphism from a finite induced substructure of $G$ into $G$ extends to an endomorphism of $G$;
\item $\MH_{\mathcal L}$ if every monomorphism from a finite induced substructure of $G$ into $G$ extends to an endomorphism of $G$;
\item $\MM_{\mathcal L}$ if every monomorphism from a finite induced subgraph of $G$ into $G$ extends to an injective endomorphism of $G$.
\end{enumerate}

We will not mention the language when it is clear from the context or when we use these properties as an adjective. Abbreviated statements like ``$G$ is \MH'' or ``$G$ is \MH-homogeneous'' mean that the structure $G$ belongs to the class $\MH_{\mathcal L}$, where $\mathcal L$ is the appropriate language (which further below will be identified with one or two finite partial orders). We will also use the general term {\em morphism extension classes} to denote these classes if they are discussed in general.

The general motivation for this paper is the unresolved relationship between morphism extension classes in the sense of inclusion. For example, we can easily see that \HH is always a subclass of \MH, because any monomorphism is also homomorphism. The other relationships between morphism extension classes can be unclear for specific classes of relational structures. More specific motivation is the relationship between \MH and \HH. In the original paper, Cameron and Ne\v{s}et\v{r}il~\cite{CameronNesetril:2006} asked to elucidate the exact relationship between these classes for undirected graphs without loops. Later Rusinov and Schweitzer~\cite{RusinovSchweitzer:2010} showed that for this class \MH=\HH. Following these results Hartman, Hubi\v{c}ka and Ma\v{s}ulovi\'{c}~\cite{Hartman:2014} provided examples of finite $L$-colored graphs that finally distinguish these classes for finite structures. $L$-colored graphs are defined as undirected graphs with sets of colors assigned to vertices as well as edges. Along with these results they have also shown that for many subclasses of $L$-colored graphs \MH and \HH coincide, which has led the authors to ask:

\begin{problem}\label{problem:mh_eq_hh_finite}~
\begin{enumerate}
\item{Do \MH and \HH coincide for infinite $P,Q$-colored graphs? (see Definition \ref{def:Lcol})}
\item{Do \MH and \HH coincide for infinite $P$-colored graphs with all vertices uncolored?}
\end{enumerate}
\end{problem}

The answers to both questions are in this paper: the classes coincide for countably infinite structures if and only if $L$ is a linear order (Theorem \ref{thm:mainthm}). 

Section \ref{sec:defs} contains the main definitions and some notations. Due to the way in which we set up our structures (the ``language" is a pair of partial orders), the reader should pay particular attention to the definition of homomorphism between colored graphs (Definition \ref{def:hom-l-graph}). Section \ref{sec:examples} contains two worked examples that illustrate the basic techniques employed in Section \ref{s:mec}, where our arguments are shown. 

\section{Studied structures}\label{sec:defs}
As mentioned in the introduction, we will be dealing with $L$-colored graphs. The original definition for an $L$-colored graph is 

\begin{definition}
Let $L$ be a partially ordered set. An $L$-colored graph is an ordered triple $(V,\chi',\chi'')$ such that $V$ is a nonempty set, $\chi':V\to L$ is an arbitrary function and $\chi'': V^2\to L$ is a function satisfying:
\begin{enumerate}
	\item{$\chi''(x, x) = 0$, and}
	\item{$\chi''(x, y) = \chi''(y, x)$.}
\end{enumerate}
A multicolored graph is an $L$-colored graph, where $L$ is the powerset of some set with set inclusion as the ordering relation. 
\end{definition}

It is implicit in the usage of this definition that $L$ is partitioned in two classes: one for colors that can be used on vertices, and the other for colors that can be assigned to edges. We will use a slightly modified definition that makes explicit this separation, corresponding to the definition in \cite{Hartman:2014} of multicolored graphs, with a minor (and innocuous) difference: instead of assigning sets of colors to the vertices and edges and using the standard definition of homomorphism, we will assign elements from partial orders with the convention that a homomorphism can map a vertex (edge) of color $c$ to a vertex (edge) of color $c'$ whenever $c\leq c'$. One can easily see the equivalence of these two definitions: in the original one, the coloring maps take values in the partially ordered set $\mathcal P(L)$, where $L$ is a set of colors, so it matches our definition; on the other hand, we can convert a graph colored with partial orders $P$ (for vertices) and $Q$ (for edges) into a structure of the type from \cite{Hartman:2014} by assigning to each vertex (edge) the set of colors $\{c': c'\leq c\}$, where $c$ is the color that the vertex (edge) takes. The formal definitions are:

\begin{definition}\label{def:Lcol}
A $P,Q$-colored graph is a tuple $(V,P,Q,\chi,\xi)$, where $V$ is a vertex set, $P$ and $Q$ are two disjoint finite partially ordered sets, $\chi:V\to P$ is an arbitrary function, and $\xi: V^2\to Q$ is a symmetric function with $\xi(v,v)=\mathbb 0$ for all $v\in V$. Our partial orders are always finite and have a minimum element $\mathbb 0$ (corresponding to uncolored vertices and nonedges).

We will say that a $P,Q$-colored graph $M$ is vertex-uniform if $\chi$ is constant.
\end{definition}

\begin{definition}\label{def:hom-l-graph}
A homomorphism between $(G,P,Q,\chi,\xi)$ and $(H,P,Q,\chi',\xi')$ is a function $f:G\to H$, such that for all $v\in G$, $$\chi(v)\leq_P\chi'(f(v))$$ and for all pairs $\{x,y\}\in G^2$, $$\xi(x,y)\leq_Q\xi'(f(x),f(y)).$$ 
\end{definition}

Note that in Definition~\ref{def:hom-l-graph} we require the same ``language," i.e., the same pair of partial orders, in both structures. We will use symbols such as $\MH_Q$ and $\MH_{P,Q}$ to denote the class of $Q$- or $P,Q$-colored graphs that are \MH. We adopt the convention that if only one partial order is mentioned, then the structures in the class have uncolored vertices.

In the course of our proofs, we will often need to consider the relations that a vertex satisfies with respect to the elements of some finite set and the conditions for extensibility of a homomorphism to a given vertex. To facilitate the discussion, we reserve the symbols $\varphi_{x,F}$ and $\zeta_{f,c}$ defined below.

\begin{definition}~
\begin{enumerate}
\item{Let $v$ be a vertex in an $L$-colored graph $M$ with edge coloring given by $\xi$, and $A$ a finite subset of $M\setminus\{v\}$. The \emph{diagram} of $v$ over $A$ is the function $$\varphi_{v,A}:A\to L$$ given by $$\varphi_{v,A}(u)=\xi(v,u).$$ }
\item{Suppose that $f:A\to B$ is a finite surjective monomorphism and $c\notin A$. Define $\zeta_{f,c}:B\to L$ by $$\zeta_{f,c}(b)=\varphi_{c,A}(f^{-1}(b)).$$}
\item{Given two functions $\varphi,\varphi':A\to L$, we will write $\varphi\preceq\varphi'$ if for all $a\in A$, $\varphi(a)\leq\varphi'(a)$}
\end{enumerate}
\end{definition}

Observe that for an $L$-colored graph, where $L$ is composed of partial orders for vertices $P$ and edges $Q$, a surjective finite monomorphism $f:A\to B$ can be extended to a vertex $c\notin A$ iff there exists $d$ such that $\zeta_{f,c}\preceq\varphi_{d,B}$ and $\chi(c)\leq_P\chi(d)$.

\section{Two worked examples}\label{sec:examples}
In this section we introduce the techniques that we will employ in the main proofs by means of two basic examples. 

\begin{definition} 
For any $n\geq 1$, $F_n$ is the partial order consisting of an antichain of size $n$ and a minimum element $\mathbb 0$.
\end{definition}

\begin{definition}
A partial order $(P,<)$ is a directed (downwards-directed) set if for any pair of elements $p,q\in P$ there exists $r\in P$ such that $p<r,q<r$ ($r<p,r<q$). 
\end{definition}

Our partial orders are assumed to have a minimum element $\mathbb 0$, and so they are automatically downwards-directed sets.

From this point on, $\mathcal C_n$ and $\Rad_n$ will denote, respectively, the class of finite graphs with edges colored by $F_n$ and its \Fraisse limit ($\mathcal C_n$ is a \Fraisse class with free amalgamation). Consider the following property:\newline

\hspace{-0.5cm}\begin{tabular}{ l l }
  \trgn{n} & 
  \begin{minipage}{0.89\textwidth}
  If $G_1,\ldots,G_{n+1}$ are finite disjoint subsets of $\Rad_n$, then there exists $x\in\Rad_n\setminus G_{n+1}$ such that each vertex of $G_i$ is related to $x$ by an edge of color $c_i$ for $1\leq i\leq n$, and for each vertex $y$ of $G_{n+1}$, the pair $xy$ is a nonedge.
  \end{minipage}
\end{tabular}\newline

\noindent Note that \trgn{n} is equivalent to the following statement:
\begin{center}
For all finite $A\subset\Rad_n$ and all functions $\varphi:A\to F_n$, $\mathcal\Rad_n\models\exists x(\varphi_{x,A}=\varphi)$,
\end{center}

\begin{fact}\label{lem:extprop}~
\begin{enumerate}
\item{$\Rad_n$ is homogeneous and every countable $F_n$-colored graph embeds into $\Rad_n$ as an induced subgraph.}
\item{$\Rad_n$ is the unique countable graph satisfying \trgn{n}, up to isomorphism.}
\end{enumerate}
\end{fact} 

\begin{example}[\Fraisse limits that are \MH- but not \HH-homogeneous]\hfill\label{ex:frlim}
We claim that $\Rad_n$ is \MH- but not \HH-homogeneous for $n\geq2$. The failure of \HH-homogeneity is easy to see, as for any nonedge $uv$ there is some $a$ such that $(a,u)$ and $(a,v)$ get different colors. Clearly, the homomorphism $u\mapsto u, v\mapsto u$ cannot be extended to $a$.
\MH-homogeneity follows easily from the extension property \trgn{n} above: given a surjective finite monomorphism $f:A\to B$ and $c\notin A$, use \trgn{n} to find a vertex $a'$ with $\varphi_{a',B}=\zeta_{f,c}$, and add the pair $(a,a')$ to $f$.
\end{example}

We will modify this example in Lemma \ref{thm:lattice} to prove that the class $\MH_Q$ contains elements that are not in $\HH_Q$ if $Q$ is not a directed set. 

\begin{definition}\label{def:diamond}
Given $n\geq 2$, $D_n$ is the partial order consisting of a finite antichain of size $n$, a minimum element $\mathbb 0$, and a maximum element $\mathds 1$. 
\end{definition}

In \cite{Hartman:2014}, Hartman, Hubi\v cka and Ma\v sulovi\'c proved that in the class of finite vertex-uniform graphs $\MH_{D_n}=\HH_{D_n}$. The following example shows that we do not have such equality for infinite graphs.

\begin{example}[An \MH but not \HH graph colored by a diamond]\label{ex:ex2}\hfill\\
We will now describe a structure $M$ colored by $D_2$ that is \MH but not \HH.
\begin{enumerate}
\item{Start with a countably infinite clique of color $\mathds 1$ partitioned into six countably infinite cliques $M_x^0, M_x^1$ for $x\in\{a,b,c\}$; for simplicity, use $M_x$ to denote $M_x^0\cup M_x^1$.}
\item{Add three new vertices $a,b,c$ and connect $x \in \{a,b,c\}$ to $M_x$ with edges of color $\mathds 1$.}
\item{Connect with color $R$ the cliques $M_a^0$ to $b$, $M_b^0$ to $c$, $M_c^0$ to $a$, and all other edges from a clique with superindex 0 to an element of $\{a,b,c\}$ with color $B$. The colors are reversed for the $M_x^1$, that is, if $M_x^0$ is connected to $y$ by color $R$, then $M_x^1$ is connected to $y$ in color $B$ ($x\in\{a,b,c\}, y\in\{a,b,c\}\setminus\{x\}$).}
\item{There are no other edges in $M$.}
\end{enumerate}
Note that $\{a,b,c\}$ forms an independent set and contains all the nonedges in $M$.  
\begin{center}
\includegraphics[scale=0.8]{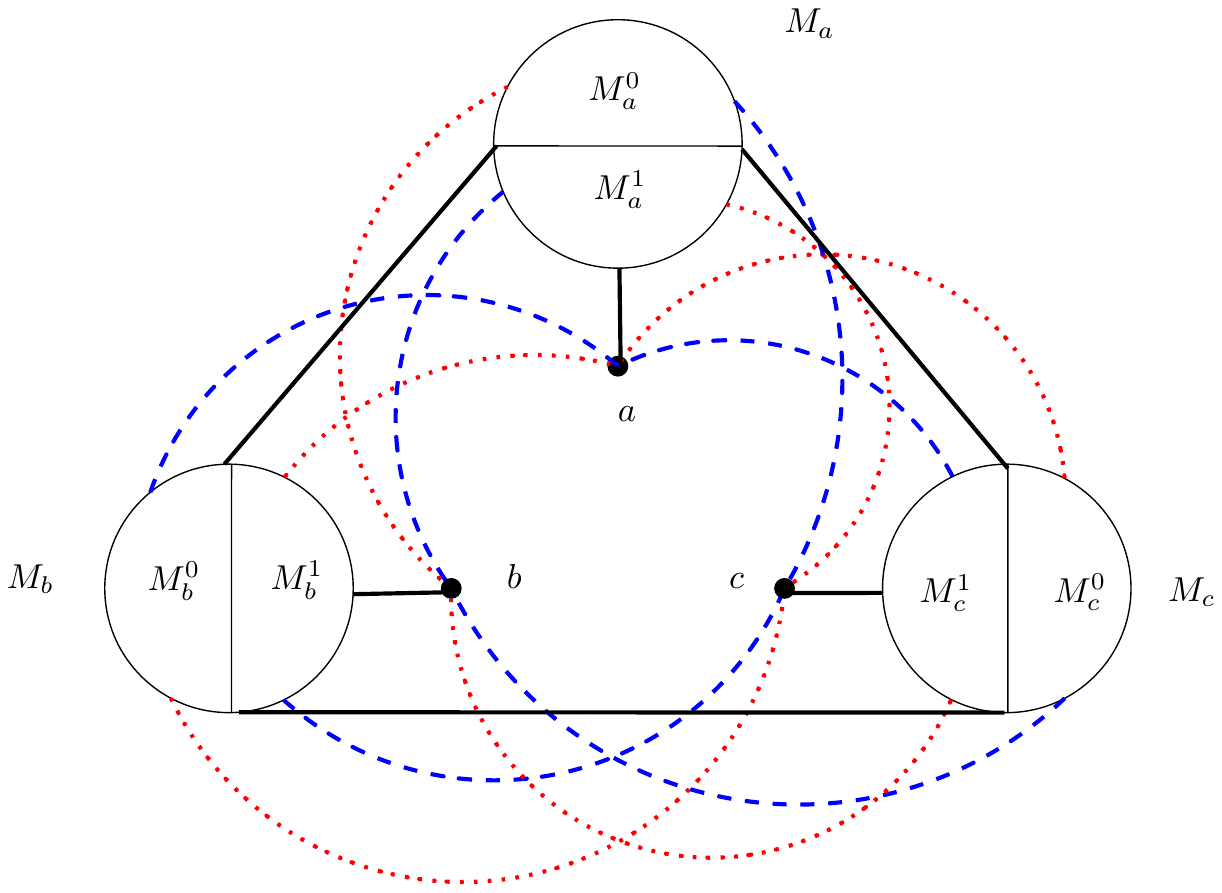}
\captionof{figure}{The structure $M$. Structures $M_x$ where $x\in\{a,b,c\}$ represents cliques in color $\mathds 1$. Edges between a vertex and a substructure or between two substructures represent complete bipartite graph. Solid black lines represent edges of color $\mathds 1$, dashed lines represent color $B$ and dotted lines color $R$.}
\end{center}
We will use the function $\pi:M_a\cup M_b\cup M_c\to\{a,b,c\}$ sending a vertex $v$ to the unique element of $\{a,b,c\}$ to which it is connected by an edge of color $\mathds 1$. Also, define $s(v)$ as 0 if $v\in M^0_{\pi(v)}$ and 1 if $v\in M^1_{\pi(v)}$.

\begin{proposition}
$M$ is \MH and not \HH.
\end{proposition}
\begin{proof}
To see that $M$ is not \HH, note that the homomorphism
\[
a\mapsto a, b\mapsto a, c\mapsto c
\]
cannot be extended to any $d\in M_c^0$, since the image of $d$ would have to be a common neighbor of $a$ and $c$ in color $\mathds 1$.

Now we prove that $M$ is \MH. Suppose that $f:H\to H'$ is a surjective monomorphism between finite substructures of $M$. Write $H_a,H_b,H_c$ for $H\cap M_a, H\cap M_b, H\cap M_c$, and similarly for $H'_a$ etc.

We can eliminate many cases by noting (1) that for every finite substructure $A$ containing at most one vertex from $\{a,b,c\}$ there are infinitely many vertices satisfying a constant cone of color $\mathds 1$ over $A$, so all homomorphisms with this type of image can be extended, and (2) that if $f\upharpoonright\{a,b,c\}$ is a bijection, then the sets $H_x$ are mapped into $M_{f(x)}$ for $x\in\{a,b,c\}$, so those homomorphisms can also be extended easily.

We may therefore assume that the image of $f$ contains exactly two elements from $\{a,b,c\}$, without loss $b$ and $c$. We will focus on the effect of $f$ on $\{a,b,c\}$. Consider $d\notin H$ and observe that $\mathds 1(d,k)$ for all $k\in H\setminus\{a,b,c\}$, and $R(d,x),B(d,y)$, where $\{x,y\}=\{a,b,c\}\setminus\{\pi(d)\}$; if $f$ maps
\begin{enumerate}
\item{$b\mapsto b, c\mapsto c$ and $a$ is mapped to some point in $M\setminus\{a,b,c\}$. Then $d$ can stay in $M^{s(d)}_{\pi(d)}$, and this extends $f$.}
\item{$b\mapsto c, c\mapsto b$ and $a$ is mapped to some point in $M\setminus\{a,b,c\}$. Then move $d$ to a vertex in $M^{1-s(d)}_{f(\pi(d))}$ if $\pi(d)\in\{b,c\}$ and to a vertex in $M^{1-s(d)}_{a}$ otherwise.}
\item{$a\mapsto c, b\mapsto b$ and and $c$ is mapped to some point in $M_c\cup M_a$. Then move $d$ to a vertex in $M_{f(\pi(d))}^{1-s(d)}$ if $\pi(d)\in\{a,b\}$, and leave it in $M_{\pi(d)}^{s(d)}$ otherwise.}
\item{$a\mapsto c, b\mapsto b, c\mapsto v\in M_b$. Then move $d$ to a vertex in $M_c^{1-s(d)}$ if $\pi(d)=a$, leave it in $M_{\pi(d)}^{s(d)}$ if $\pi(d)=c$, and move it to some vertex in $M_b^{1-s(d)}$ if $\pi(d)=b$.}
\item{$a\mapsto c, c\mapsto b, b\mapsto v\in M_a\cup M_b$. Then move $d$ to a vertex in $M^{s(d)}_{f(\pi(d))}$ if $\pi(d)\in\{a,c\}$, and to a vertex in $M^{1-s(d)}_b$ otherwise.}
\item{$a\mapsto c, c\mapsto b, b\mapsto v\in M_c$.Then move $d$ to a vertex in $M^{1-s(d)}_{f(\pi(d))}$ if $\pi(d)\in\{b,c\}$, and to $M^{s(d)}_{c}$ otherwise.}
\end{enumerate}
\end{proof}
\end{example}

A small modification of this example will allow us to prove in Theorem \ref{thm:mainthm} that if $Q$ contains incomparable elements, then $\MH_Q\neq\HH_Q$.

\section{$\HH_{P,Q}=\MH_{P,Q}$ iff $Q$ is a linear order}\label{s:mec}

We can further exploit Example \ref{ex:frlim} to produce an \MH but not \HH vertex-uniform graph that is connected in every maximal color, for any $Q$ that is not a directed set, even when vertex colorings are allowed.

\begin{lemma}\label{lem:partRado}
For any $m\geq 1$, there exists a partition of $\Rad_n$ into $m$ classes $C_1,\ldots, C_m$ such that:
\begin{enumerate}
\item{each $C_i$ is isomorphic to $\Rad_n$, and}
\item{for any finite $A\subset\Rad_n$, $\psi:A\to F_n$, and $k\leq m$, there exists $x\in C_k$ such that $\varphi_{x,A}=\psi$.}
\end{enumerate}
\end{lemma}
\begin{proof}
Start with a countably infinite set $X$ partitioned into $m$ infinite classes $C_1,\ldots, C_m$. We will show how to impose a structure on $X$ satisfying the conditions above.

Enumerate the finite nonempty subsets of $X$ as $\{Y_i:i\in\omega\}$, and for each of them let $\{t_j^i:1\leq j\leq(n+1)^{|Y_i|}\}$ be an enumeration of all the functions $Y_i\to F_n$. 

\noindent\underline{Step 0:} Choose a vertex $v_{1,r}^0$ in each class $C_r$, and insert edges in such a way that $\varphi_{v_{1,r}^0,Y_0}=t_1^0$. Continue in the same way for $j\in\{2,\ldots,(n+1)^{|Y_i|}\}$, being careful to select fresh vertices from each class at each step.

\noindent\underline{Step $s+1$:} Suppose that we have carried out the construction up to step $s$, corresponding to the functions over $Y_s$. Consider now $Y_{s+1}$ and the $t_{j}^{s+1}$. As before, insert edges from fresh vertices from each class to $Y_{s+1}$ as prescribed by the $t_j^{s+1}$.

At each point in the construction we have used only finitely many elements from each class, so the next step can always be carried out. Let $G$ be the structure on $X$ with all the edges that are eventually introduced by the process we have described. By construction, $G$ and the structures induced by $G$ on each $C_i$ satisfy \trgn{n}, and are therefore isomorphic to $\Rad_n$ by Fact \ref{lem:extprop}. Given a finite subset $A\subset G$ and $\psi:A\to F_n$, there is some step when we introduced edges from a vertex $x_i\in C_i$ to $A$ with $\varphi_{x_i,A}=\psi$ for each $i\in\{1\ldots,n\}$, so the second condition is also satisfied.
\end{proof}

\begin{lemma}\label{thm:lattice}
Let $P$ and $Q$ be finite partially ordered sets. If $\MH_{P,Q}=\HH_{P,Q}$, then $Q$ is a directed set.
\end{lemma}
\begin{proof}
We will prove the contrapositive. Supose that $Q$ has $n\geq 2$ maximal elements $R_1,\ldots, R_n$; let $m$ denote $|P|$, and let $M_0$ be the \Fraisse limit of graphs colored by $F_n$. Partition $M_0$ into $m$ copies $C_1,\ldots,C_m$ of itself satisfying the conditions from Lemma \ref{lem:partRado}.

Enumerate the elements of $P$ as $e_1,\ldots,e_m$ and assign color $e_i$ to the elements of $C_i$. Now we carry on inserting edges of non-maximal colors in any way and leaving at least one nonedge of equal-colored vertices. We therefore have a substructure consisting of two equal-colored vertices with a nonedge between them and linked to a third vertex by different edge-colors without a common upper bound. The extension axioms \trgn{n} for maximal colors guarantee monomorphism-homogeneity and the three-vertex substructure witnesses a failure of ho\-mo\-morphism-homogeneity.
\end{proof}

The proof of the following result is a generalization of the construction of Example \ref{ex:ex2}.
\begin{theorem}\label{thm:mainthm}
Let $P$ and $Q$ be finite partially ordered sets. $\MH_{P,Q}=\HH_{P,Q}$ iff $Q$ is a linear order.
\end{theorem}
\begin{proof}
Suppose first that the set of edge-colors is linearly ordered and we have an arbitrary $P,Q$-structure $G$ that is $\MH_{P,Q}$. Let $f:H\to H'$ be a surjective homomorphism, and fix an enumeration $H_1,\ldots,H_k$ of the partition of $H$ induced by $f$ (so the $H_i$ are preimages of a single vertex). Let $S$ be a set of all transversals to the partition $H=\bigcup\{H_i:1\leq i\leq k\}$. Given a vertex $u\notin H$ we can use the linear order on $Q$ to find a transversal $s^0$ such that, for all $s\in S$, $$\varphi_{u,s}\preceq\varphi_{u,s^0}.$$

Each function of the form $f\upharpoonright s$ with $s\in S$ is a monomorphism, so it can be extended as a homomorphism to $u$ by $\MH_{P,Q}$-homogeneity. Note that each $\varphi_{u,H}\upharpoonright s$ is realized in $G$ (by $u$, for example), so extend $f\upharpoonright s^0$ to another realization $u_0$ of $s^0$. We have $\varphi_{u,s}\preceq\varphi_{u,s^0}$, and thus $f\cup\{(u,(f\upharpoonright s^0)(u_0))\}$ is a homomorphism.
\vspace{0.5cm}

Now we will prove that if $Q$ is not a linear order, then there exists a \MH structure $M$ that is not \HH. We will construct an \MH structure $M$, connected in $\mathds 1$ and each of colors $P_i$ directly below $\mathds 1$, that fails to be \HH. The stucture we construct will be vertex-uniform, but if vertex colorings are allowed, we can modify it, giving some maximal color to the special vertices $x_1,\ldots,x_{n+1}$ introduced below and lower or incomparable colors to the other vertices, thus ensuring that the image of $x_i$ is some other special vertex (this actually simplifies the argument).

By Lemma \ref{thm:lattice}, we know that $Q$ is a finite directed set, so it has a top element $\mathds 1$. Let $P_1,\ldots,P_n$ be the top elements of $Q\setminus\{\mathds 1\}$. 

For each $i\in\{1,\ldots,n+1\}$, let $M_i$ be a copy of the Rado graph in color $\mathds 1$ partitioned into $n!$ classes and satisfying the conditions from Lemma \ref{lem:partRado}, with all nonedges filled by predicates from $Q$ in such a way that all predicates are used; this defines the edge-coloring function $\xi$ within each of the $M_i$. Let $x_1,\ldots,x_{n+1}$ be new vertices; now we will explain how to extend $\xi$ to the rest of $((\bigcup\{M_i:1\leq i\leq n+1\})\cup\{x_1,\ldots,x_{n+1}\})^2$.

If $v\in M_i$ and $w\in M_j$ with $i\neq j$, declare $\xi(v,w)=\mathds 1$. Also, for all $v\in M_i$ declare $\xi(v,x_i)=\mathds 1$.

Index the parts of the partition of $M_i$ as $M_i^\sigma$, where indices $\sigma\in S_n$ are elements of symmetric group corresponding to a partition of $M_i$ into $n!$ classes, so $M_i=\bigcup_{\sigma\in S_n}M_i^\sigma$. Define $c_i:\{x_1,\ldots,x_{i-1},x_{i+1},\ldots,x_{n+1}\}\to\{1,\ldots,n\}$ by $$c_i(x_j)=\begin{cases} j &\mbox{if } j\leq i-1\\ j-1 &\mbox{if } j\geq i+1\end{cases}$$ and declare for $v\in M_i^\sigma$ and $j\neq i$, $\xi(v,x_j)=P_{\sigma(c_i(x_j))}$. There are no other edges in $M$.

It is easy to see that $M$ is not \HH. Note that $x_i$ and $x_j$ are either equal or at distance 3 in the predicate $\mathds 1$. Choose distinct $j,k,\ell\in\{1,\ldots, n+1\}$ and $v\in M_j^\sigma$ for any $\sigma\in S_n$. The local homomorphism $x_k\mapsto x_k, x_\ell\mapsto x_k, x_j\mapsto x_j$ cannot be extended to $v$ because the only element above both $P_{\sigma(c_j(x_k))}$ and $P_{\sigma(c_j(x_\ell))}$ is $\mathds 1$, but $x_k$ and $x_j$ do not have a common neighbor in $\mathds 1$.

Now we need to show that $M$ is \MH. Suppose that $f:H\to K$ is a surjective monomorphism between finite substructures of $M$. If $|K\cap\{x_1,\ldots,x_{n+1}\}|\leq 1$, then it is easy to extend the monomorphism to any new vertex $v$ because there are infinitely many elements with all edges of color $\mathds 1$ towards $K$, so we can assume that the image of $f$ contains at least two of the special vertices. Since the pairs $x_i,x_j$ with $i\neq j$ are the only nonedges with distinct endpoints in $M$, we know that if $x_i,x_j\in K$ then their preimages under $f$ are also contained in $\{x_1,\ldots,x_k\}$. 

\begin{claim}\label{claim:exts}
Given any $F\subset M\setminus\{x_1,\ldots,x_{n+1}\}$, $S\subseteq\{x_1,\ldots,x_{n+1}\}$ and injective $t:S\to\{\mathds 1,P_1,\ldots,P_{n}\}$, there exists some $v\in M$ that is connected to all of $F$ by edges of type $\mathds 1$ and satisfies $t$ over $S$, i.e., $\varphi_{v,S}=t$.
\end{claim}
\begin{proof}[Proof of Claim \ref{claim:exts}]
Note that for any finite $F\subset M\setminus\{x_1,\ldots,x_{n+1}\}$ there exists $v_i^\sigma\in M_i^\sigma$ for each $i\in\{1,\ldots,n\}$ and $\sigma\in S_n$ connected to $F$ by edges of type $\mathds 1$. Given any $F, i,\sigma$, we can find $v_i^\sigma$ that is connected by $\mathds 1$ to $F\cap M_i$ by Lemma \ref{lem:partRado}. By the construction of $M$, this vertex will be connected by $\mathds 1$ to the rest of $F$. 

If $t$ takes value $\mathds 1$ in $x_i$, then we need to find a $\sigma$ such that an element of $M_i^\sigma$ is connected to $S\setminus\{x_i\}$ in the correct manner. But this always happens, since all possible connections to $S\setminus\{x_i\}$ appear in $M_i$, since all permutations are present.

And if $\mathds 1$ is not in the image of $t$, then we are free choose any $i$ and find the appropriate $\sigma$.
\end{proof}

Using the Claim above, it is easy to prove that $M$ is \MH. Let $S=K\cap\{x_1,\ldots,x_{n+1}\}$, $t(x_i)=\xi(v,f^{-1}(x_i))$ and $F=K$. A vertex $w$ satisfying $t$ over $S$ and connected by $\mathds 1$ to $F$ extends $f:H\to K$ to the vertex $v\notin H$ as a homomorphism.
\end{proof}

\section*{Acknowledgments}
Andr\'es Aranda received funding from the ERC under the European Union's Horizon 2020 research and innovation programme (grant agreement No 681988, CSP-Infinity).

David Hartman is supported by project P202/12/G061 of the Czech Science Foundation (GA\v{C}R).


\end{document}